\documentclass[12pt]{amsart}
\usepackage{amssymb,color}
\numberwithin{equation}{section} 

\usepackage{graphicx}

\newcommand{\bea}{\begin{eqnarray}}
\newcommand{\eea}{\end{eqnarray}}
\newcommand{\ba}{\begin{array}}
\newcommand{\ea}{\end{array}}
\newcommand{\edc}{\end{document}}
\newcommand{\bc}{\begin{center}}
\newcommand{\ec}{\end{center}}
\newcommand{\be}{\begin{equation}}
\newcommand{\ee}{\end{equation}}

\def\bc{{\mathbb C}}

\def\m{\mu}

\newtheorem{thm}{Theorem}[section]
\newtheorem{lem}[thm]{Lemma}
\newtheorem{cor}[thm]{Corollary}

\newtheorem{prop}[thm]{Proposition}
\newtheorem{defin}[thm]{Definition}
\theoremstyle{remark}
\newtheorem{rem}{Remark}[section]
\newtheorem{ex}{Example}[section]

\begin{document}

\title{Hypercyclic and supercyclic linear operators on non-Archimedean vector spaces}


\author{Farrukh Mukhamedov}
\address{Farrukh Mukhamedov\\
Department of Mathematical Sciences\\
College of Science, The United Arab Emirates University\\
P.O. Box, 15551, Al Ain\\
Abu Dhabi, UAE} \email{{\tt far75m@gmail.com} {\tt
farrukh.m@uaeu.ac.ae}}

\author{Otabek Khakimov}
\address{Otabek Khakimov\\
Institute of mathematics, National University of Uzbekistan, 29,
Do'rmon Yo'li str., 100125, Tashkent, Uzbekistan.} \email {{\tt
hakimovo@mail.ru}}

\begin{abstract}
Main objective of the present paper is to develop the theory of
hypercyclicity and supercyclicity of linear operators on
topological vector space over non-Archimedean valued fields. We
show that there does not exist any hypercyclic operator on finite
dimensional spaces. Moreover, we give sufficient and necessary
conditions of hypercyclicity (resp. supercyclicity) of linear
operators on separable $F$-spaces. It is proven that a linear
operator $T$ on topological vector space $X$ is hypercyclic
(supercyclic) if it satisfies Hypercyclic (resp. Supercyclic)
Criterion. We consider backward shifts on $c_0$, and characterize
hypercyclicity and supercyclicity of such kinds of shifts.
Finally, we study hypercyclicity, supercyclicity of operators
$\lambda I+\mu B$, where $I$ is identity and $B$ is backward
shift. We note that there are essential differences between the
non-Archimedean and real cases.

\end{abstract}

\maketitle

\section{introduction}

Linear dynamics is a young and rapidly evolving branch of functional
analysis, which was probably born in 1982 with the Toronto Ph.D.
thesis of C. Kitai \cite{K}. It has become rather popular, thanks to
the efforts of many mathematicians (see \cite{GS,GE}). In
particular, hypercyclicity and supercyclicity of weighted bilateral
shifts were characterized by Salas \cite{S1,S2}. In \cite{Shk1} it
has been proved that there exists a bounded linear operator $T$
satisfying the Kitai Criterion on separable infinite dimensional
Banach space. For more detailed information about cyclic,
hypercyclic linear operators we refer to \cite{B}.

We stress that all investigations on dynamics of linear operators
were considered over the field of the real or complex numbers. On
the other hand, non-Archimedean functional analysis is
well-established discipline, which was developed in Monna's series
of works in 1943. Last decades there have been published a lot of
books devoted to the non-Archimedean functional analysis (see for
example \cite{PG,Sch}). A main objective of the present paper is to
develop the theory of hypercyclicity and supercyclicity of linear
operators on a topological vector space over non-Archimedean valued
fields. In section 3, we will show that there does not exist any
hypercyclic operator on a finite dimensional space. Moreover, we
give sufficient and necessary conditions of hypercyclicity
(supercyclicity) of linear operators on separable $F$-spaces Theorem
\ref{Btthm} (resp. Theorem \ref{Btthm2}). We will show that a linear
operator $T$ on topological vector space $X$ is hypercyclic
(supercyclic) if it satisfies Hypercyclic (resp. Supercyclic)
Criterion.  Note that the shift operators have many applications in
many branches of modern mathematics (in real setting). But if one
considers this type of operators in a non-Archimedean setting, it
turns out that the shift operators have certain applications in
$p$-adic dynamical systems \cite{Je,KL}. Therefore, in section 4 we
consider backward shifts on $c_0$, and characterize hypercyclicity
and supercyclicity of such kinds of operators. In section 5 we will
consider an operator $\lambda I+\mu B$, where $I$ is identity and
$B$ is backward shift. We prove that the operator $I+\mu B$ cannot
be hypercyclic while in the real case this operator is hypercyclic
when $|\m|>1$. This is an essential difference between the
non-Archimedean and real cases.

\section{definations and preliminary results}

All fields appearing in this paper are commutative. A valuation on a field $\mathbb K$ is
a map $|\cdot|:\mathbb K\to[0,+\infty)$ such that:\\
$(i)$ $|\lambda|=0$ if and only if $\lambda=0$,\\
$(ii)$ $|\lambda\mu|=|\lambda|\cdot|\lambda|$ (multiplicativity),\\
$(iii)$ $|\lambda+\mu|\leq|\lambda|+|\mu|$ (triangle inequality),
for all $\lambda,\mu\in\mathbb K$. The pair $(\mathbb K,|\cdot|)$ is called a valued field.
We often write $\mathbb K$
instead of $(\mathbb K,|\cdot|)$.
\begin{defin}
Let $\mathbb K=(\mathbb K,|\cdot|)$ be a valued field. If $|\cdot|$
satisfies the strong triangle inequality: $(iii')$
$|\lambda+\mu|\leq|\max\{|\lambda|,|\mu|\}$, for all
$\lambda,\mu\in\mathbb K$, then  $|\cdot|$ is called
non-Archimedean, and $\mathbb K$ is called a non-Archimedean valued
field
\end{defin}

\begin{rem} In what follows, we always assume a norm in non-Archimedean valued
field is nontrivial.
\end{rem}

From the strong triangle inequality we get the following useful
property of non-Archimedean value: If $|\lambda|\neq|\mu|$ then
$|\lambda\pm\mu|=\max\{|\lambda|,|\mu|\}$. We frequently use this
property, and call it as the non-Archimedean norm's property. A
non-Archimedean valued field $\mathbb K$ is a metric space and it is
called {\it ultrametric space}.

Let $a\in\mathbb K$ and $r>0$. The set
$$
B(a,r):=\{x\in\mathbb K: |x-a|\leq r\}
$$
is called the {\it closed ball with radius $r$ about $a$}. (Indeed, $B(a,r)$ is closed in the induced topology).
Similarly,
$$
B(a,r^-):=\{x\in\mathbb K: |x-a|<r\}
$$
is called the {\it open ball with radius $r$ about $a$}.

We set
$|\mathbb K|:=\{|\lambda|: \lambda\in\mathbb K\}$ and $\mathbb K^\times:=\mathbb K\setminus\{0\}$,
the {\it multiplicative group} of $\mathbb K$. Also,
$|\mathbb K^\times|:=\{|\lambda|: \lambda\in\mathbb K^\times\}$ is a multiplicative group
of positive real
numbers, the {\it value group} of $\mathbb K$. There are two possibilities:
\begin{lem}\label{absvalsubgr}$($\cite{Sch}$)$
Let $\mathbb K$ be a non-Archimedean valued field. Then the value
group of $\mathbb K$ either is dense or is discrete; in the latter
case there is a real number $0<r<1$ such that $|\mathbb
K^\times|=\{r^s: s\in{\mathbb Z}\}$.
\end{lem}

\begin{defin} A pari $(E,\|\cdot\|)$ is called a $\mathbb K$-normed space over $\mathbb K$, if $E$ is a $\mathbb K$-vector
space and  $\parallel\cdot\parallel: E\to[0,+\infty)$ is a
non-Archimedean norm, i.e.
such that:\\
$(i)$ $\parallel\textbf{x}\parallel=\textbf{0}$ if and only if $\textbf{x}=\textbf{0}$,\\
$(ii)$ $\parallel\lambda\textbf{x}\parallel=|\lambda|\parallel\textbf{x}\parallel$,\\
$(iii)$ $\parallel\textbf{x}+\textbf{y}\parallel\leq\max\{\parallel\textbf{x}\parallel, \parallel\textbf{y}\parallel\}$,
for all $\textbf{x}, \textbf{y}\in E,\ \lambda\in\mathbb K$.
\end{defin}

 We frequently write $E$ instead of $(E,\parallel\cdot\parallel)$. $E$
is called a $\mathbb K$-Banach space or a Banach space over $\mathbb
K$ if it is complete with respect to the induced ultrametric
$d(\textbf{x}, \textbf{y})=\parallel\textbf{x}-\textbf{y}\parallel$.
\begin{ex} Let $\mathbb K$ be a non-Archimedean valued field; then
$$
l_\infty:=\mbox{all bounded sequences on }\mathbb K
$$
with pointwise addition and scalar multiplication and the norm
$$
\parallel\textbf{x}\parallel_\infty:=\sup_{n}|x_n|
$$
is a $\mathbb K$-Banach space.
\end{ex}

\begin{rem}
From now on we often drop the prefix "$\mathbb K$"- and write vector
space, normed space, Banach space instead of $\mathbb K$-vector
space, $\mathbb K$-normed space, $\mathbb K$-Banach space,
respectively.
\end{rem}

In what follows, we need the following auxiliary fact.

\begin{lem}\label{x_ny_nlambda_n}
Let $E$ be a normed space over a non-Archimedean valued field
$\mathbb K$. Then for each pair of sequences $(\textbf{x}_n)$ and
$(\textbf{y}_n)$ in $E$ such that
$\parallel\textbf{x}_n\parallel\cdot\parallel\textbf{y}_n\parallel\to0$
as $n\to\infty$ there exists a sequence $(\lambda_n)\subset\mathbb
K^\times$ such that
\begin{equation}\label{xylambda}
\lambda_n\textbf{x}_n\to\textbf{0}\ \ \ \mbox{and}\ \ \ \lambda_n^{-1}\textbf{y}_n\to\textbf{0},\ \ \ \ \mbox{as}\ \ n\to\infty.
\end{equation}
\end{lem}

\begin{proof} First, we will prove the lemma for the case when a value group of
$\mathbb K$ is a discrete. Then according to Lemma \ref{absvalsubgr}
there exists a real number $r\in(0,1)$ such that $|\mathbb
K^\times|=\{r^s: s\in\mathbb Z\}$. Let $(n_k)$ and $(m_k)$ be the
increasing subsequences of $\mathbb{N}$ with
$(n_k)\cup(m_k)=\mathbb{N}$ such that
$$
\|\textbf{x}_{n_k}\|\cdot\|\textbf{y}_{n_k}\|=\textbf{0}, \ \
\|\textbf{x}_{m_k}\|\cdot\|\textbf{y}_{m_k}\|\neq\textbf{0}, \
\forall k
$$
Let us define $\nu_{n_k}\in\mathbb K$ as follows
$$
|\nu_{n_k}|=\left\{\begin{array}{lll}
1, & \mbox{if }\ \textbf{x}_{n_k}=\textbf{y}_{n_k}=\textbf{0};\\
\frac{\parallel\textbf{y}_{n_k}\parallel}{r^{{n_k}}}, & \mbox{if }\ \textbf{x}_{n_k}=\textbf{0},\ \textbf{y}_{n_k}\neq\textbf{0};\\[2mm]
\frac{r^{{n_k}}}{\parallel\textbf{x}_{n_k}\parallel}, & \mbox{if }\ \textbf{x}_{n_k}\neq\textbf{0},\ \textbf{y}_{n_k}=\textbf{0},\\[2mm]
\end{array}\right.
$$
Since $0<r<1$, for any $\varepsilon>0$ there exists positive integer
$k'$ such that $\|\nu_{n_k}\textbf{x}_{n_k}\|<\varepsilon$ and
$\|\nu_{n_k}^{-1}\textbf{y}_{n_k}\|<\varepsilon$ for any $k>k'$.

On the other hand, according to Lemma \ref{absvalsubgr}, there
exists a sequence of integer numbers $(\alpha_{m_k})$ such that
\begin{equation}\label{r^alpha_n}
r^{2\alpha_{m_k}}\leq\frac{\parallel\textbf{y}_{m_k}\parallel}{\parallel\textbf{x}_{m_k}\parallel}\leq
r^{2\alpha_{m_k}-2}.
\end{equation}
For any $k\geq1$ we take $\mu_{m_k}\in\mathbb K^\times$ such that
$|\mu_{m_k}|=r^{\alpha_{m_k}}$. Then from \eqref{r^alpha_n} we get
$$
\begin{array}{ll}
\parallel\mu_{m_k}\textbf{x}_{m_k}\parallel=r^{\alpha_{m_k}}\parallel\textbf{x}_{m_k}
\parallel\leq\parallel\textbf{x}_{m_k}\parallel^{\frac{1}{2}}\cdot
\parallel\textbf{y}_{m_k}\parallel^{\frac{1}{2}},\\[2mm]
\parallel\mu_{m_k}^{-1}\textbf{y}_{m_k}\parallel=r^{-\alpha_{m_k}}\parallel\textbf{y}_{m_k}
\parallel\leq r^{-1}\parallel\textbf{x}_{m_k}\parallel^{\frac{1}{2}}\cdot
\parallel\textbf{y}_{m_k}\parallel^{\frac{1}{2}},
\end{array}
$$
Since
$\parallel\textbf{x}_{m_k}\parallel\cdot\parallel\textbf{y}_{m_k}\parallel\to0$,
for any $\varepsilon>0$ there exists a positive integer $k''>0$ such
that $\|\mu_{m_k}\textbf{x}_{m_k}\|<\varepsilon$ and
$\|\mu_{m_k}^{-1}\textbf{y}_{m_k}\|<\varepsilon$ for any $k>k''$.

Finally, we define a sequence $\{\lambda_n\}$ as follows:
$$
\lambda_n= \left\{
\begin{array}{ll}
\nu_{n}, \ \ \textit{if} \ \ n\in(n_k)\\
\mu_{n},  \ \ \textit{if} \ \ n\in(m_k)
\end{array}.
\right.
$$
 Then for any $\varepsilon>0$
one has  $\|\lambda_{n}\textbf{x}_{n}\|<\varepsilon$ and
$\|\lambda_{n}^{-1}\textbf{y}_{n}\|<\varepsilon$ for any
$n>\max\{n_{k'},m_{k''}\}$.

Now, we suppose that value group of $\mathbb K$ is dense. Then we
can find sequences $(\textbf{x}'_{n})$ and $(\textbf{y}'_{n})$ such
that
$$
\|\textbf{x}'_{n}\|>\|\textbf{x}_{n}\|,\ \
\|\textbf{y}'_{n}\|>\|\textbf{y}_{n}\|, \
\|\textbf{x}'_{n}\|\cdot\|\textbf{y}'_{n}\|<\|\textbf{x}_{n}\|\cdot\|\textbf{y}_{n}\|+\frac{1}{n}
$$
 It is clear that
$\|\textbf{x}'_{n}\|\cdot\|\textbf{y}'_{n}\|\to0$ as $n\to\infty$.
Fix a $a\in\mathbb K^\times$ with $|a|>1$ Then there exists a
sequence $(\beta_n)$ such that
$$
|a|^{\beta_{n}}\leq\sqrt{\frac{\parallel\textbf{y}'_{n}\parallel}{\parallel\textbf{x}'_{n}\parallel}}\leq|a|^{\beta_{n}+1}
$$
Define a sequence $\lambda_n:=a^{\beta_n}$. Then we have
$$
\begin{array}{ll}
\|\lambda_n\textbf{x}_{n}\|<\|\lambda_n\textbf{x}'_{n}\|=|a|^{\beta_n}\|\textbf{x}'_n\|\leq\sqrt{\|\textbf{x}'_n\|\cdot\|\textbf{y}'_n\|}\\[3mm]
\|\lambda^{-1}_n\textbf{y}_{n}\|<\|\lambda^{-1}_n\textbf{y}'_{n}\|=|a|^{-\beta_n}\|\textbf{y}'_n\|\leq|a|\sqrt{\|\textbf{x}'_n\|\cdot\|\textbf{y}'_n\|}
\end{array}
$$
Since $\|\textbf{x}'_{n}\|\cdot\|\textbf{y}'_{n}\|\to0$ we get
\eqref{xylambda}. This completes the proof.
\end{proof}

Let $X$ and $Y$ be topological vector spaces over non-Archimedean
valued field $\mathbb K$. By $L(X,Y)$ we denote the set of all
continuous linear operators from $X$ to $Y$. If $X=Y$ then
$L(X,Y)$ is denoted by $L(X)$. In what follows, we use the
following terminology: $T$ is a linear continuous operator on $X$
means that $T\in L(X)$. The $T$-{\it orbit} of a vector
$\textbf{x}\in X$, for some operator $T\in L(X)$, is the set
$$
O(\textbf{x},T):=\{T^n(\textbf{x}); n\in\mathbb Z_+\}.
$$
An operator $T\in L(X)$ is called {\it hypercyclic} if there
exists some vector $\textbf{x}\in X$ such that its $T$-orbit is
dense in $X$. The corresponding vector $\textbf{x}$ is called
$T$-{\it hypercyclic}, and the set of all $T$-hypercyclic vectors
is denoted by $HC(T)$. Similarly, $T$ is called {\it supercyclic}
if there exists a vector $\textbf{x}\in X$ such that whose
projective orbit
$$
\mathbb K\cdot O(\textbf{x},T):=\{\lambda T^n(\textbf{x}); n\in\mathbb Z_+,\ \lambda\in\mathbb K\}
$$
is dense in $X$. The set of all $T$-supercyclic vectors is denoted by $SC(T)$. Finally, we recall that $T$ is called {\it cyclic}
if there exists $\textbf{x}\in X$ such that
$$
\mathbb K[T]\textbf{x}:=\mbox{span}O(\textbf{x},T)=\{P(T)x; P\ \mbox{polynomial}\}
$$
is dense in $X$. The set of all $T$-cyclic vectors is denoted by $CC(T)$.

\begin{rem}\label{remsubset}
We stress that the notion of hypercyclicity makes sense only if
the space $X$ is separable. Note that one has
$$
HC(T)\subset SC(T)\subset CC(T).
$$
\end{rem}

\begin{rem}
Note that if $T$ is a hypercyclic operator on Banach space then
$\|T\|>1$.
\end{rem}

\section{hypercyclicity and supercyclicity of linear operators}

In this section we find sufficient and necessary conditions to
hypercyclicity of linear operators on $F$-spaces. In what follows,
by {\it $F$-space} we mean a topological vector space $X$  which is
metrizable and complete over a non-Archimedean field. Basically,
this section mostly repeats the well-known facts from the dynamics
of linear operators \cite{B}. But for the sake of completeness, we
are going to prove them (with taking into account
non-Archimedeanness of the space). In this section, a main approach
is based on the Baire category theorem.

We start with the well-known equivalence between hypercyclicity and
topological transitivity: an operator $T$ acting on some separable
completely metrizable space $X$ is hypercyclic iff for each pair of
non-empty open sets $(U,V)\in X$, one can find $n\in\mathbb N$ such
that $T^n(U)\cap V\neq\varnothing$; in this case, there is in fact a
residual set of hypercyclic vectors. From this, one gets immediately
the so-called Hypercyclicity Criterion, a set of sufficient
conditions for hypercyclicity with a remarkably wide range of
applications. The analogous Supercyclicity Criterion is proved along
the same lines.

Now we show that hypercyclicity turns out to be a purely
infinite-dimensional phenomenon.

\begin{prop}\label{findim} Let $X\neq\{0\}$ be a finite-dimensional space. Then each operator $T\in L(X)$ is not hypercyclic.
\end{prop}
\begin{proof} Without loss of generality, we may assume that $X=\mathbb K^m$ for some $m\geq 1$.
Now we are going to prove that each operator $T\in L(\mathbb K^m)$
is not hypercyclic. Suppose that a linear operator $T$ on $\mathbb
K^m$ is hypercyclic. Take $\textbf{x}\in HC(T)$. The density of
$O(\textbf{x},T)$ in $\mathbb K^m$ implies that the family
$\{\textbf{x},T(\textbf{x}),\dots,T^{m-1}(\textbf{x})\}$ forms a
linearly independent system. Hence, this collection is a basis of
$\mathbb K^m$. For any $\alpha\in\mathbb K\setminus\{0\}$, one can
find a sequence of integers $(n_k)$ such that
$T^{n_k}(\textbf{x})\to\alpha\textbf{x}$. Then
$T^{n_k}(T^i\textbf{x})=T^i(T^{n_k}\textbf{x})\to\alpha
T^i(\textbf{x})$ for each $i<m$. Hence for any
$\textbf{y}\in\mathbb K^m$ we obtain $T^{n_k}(\textbf{y})\to\alpha
\textbf{y}$ which yields that $\det(T^{n_k})\to\alpha^m$, i.e.
$\det(T)^{n_k}\to\alpha^m$. Thus putting $a:=\det(T)$, we have
the set $\{a^n; n\in\mathbb N\}$ is dense in $\mathbb
K\setminus\{0\}$. But it is impossible. Indeed, it is clear that
$|a^n-z|>1$ for any $z\in\mathbb K\setminus B(0,1)$ if $|a|\leq1$
and $|a^n-w|>1$ for any $w\in B(0,1)$ if $|a|>1$.
\end{proof}
Our first characterization of hypercyclicity is a direct application of the Baire category
theorem.

\begin{thm}$($\textsc{Transitivity theorem}$)$\label{Btthm} Let
$X$ be a separable $F$-space and $T\in L(X)$. The following
statements are equivalent:
\begin{enumerate}
\item[(i)] $T$ is hypercyclic;

\item[(ii)] $T$ is {\bf topologically transitive}; that is, for
each pair of non-empty open sets $(U,V)\subset X$ there exists
$n\in\mathbb N$ such that $T^n(U)\cap V\neq\varnothing$.
\end{enumerate}
\end{thm}

\begin{proof}
(i) Assume $HC(T)\neq\varnothing$. Since $X$ has no isolated
points, for any $k\in\mathbb N$ it is easy to see that
$T^k(\textbf{x})\in HC(T)$ if and only if $\textbf{x}\in HC(T)$.
Let $U,V$ be open sets in $X$. Take $\textbf{x}\in U\cap HC(T)$.
Then there exists a number $n\in\mathbb N$ such that
$T^n(\textbf{x})\in V$. This means that $T$ is topologically
transitive.

(ii) Let $T$ be topologically transitive and $\{V_k\}_{k\in\mathbb
N}$ be a countable basis of open sets on $X$ (this kind of basis
exists since $X$ is a separable $F$-space). Then from the
topological transitivity of $T$, for any $k\geq1$ and non-empty open
set $U\subset X$ there exists an $n$ such that $U\cap
T^{-n}(V_k)\neq\varnothing$. This means that each open set
$\bigcup\limits_{n\geq0}T^{-n}(V_k)$ is dense, hence one gets the
density of
$\bigcap\limits_{k\geq1}\bigcup\limits_{n\geq0}T^{-n}(V_k)$. On the
other hand, we have
\begin{equation}\label{HC(T)}
HC(T)=\bigcap_{k\geq1}\bigcup_{n\geq0}T^{-n}(V_k).
\end{equation}
Consequently, $HC(T)\neq\varnothing$. This completes the proof.
\end{proof}

\begin{cor}\label{corBt}
Let
$X$ be a separable $F$-space and $T\in L(X)$. If $T$ is hypercyclic then $HC(T)$ is $G_\delta$ set.
\end{cor}

\begin{proof}
According to Theorem \ref{Btthm} hypercyclicity of $T$ implies its
topological transitivity. From \eqref{HC(T)} one easily sees that
$HC(T)$ is $G_\delta$ set.
\end{proof}

\begin{defin} Let $X$ be a topological vector space, and let $T\in L(X)$. It is said that $T$ satisfies the
{\bf Hypercyclic Criterion} if there exist an increasing sequence of integers $(n_k)$, two dense sets
$\mathcal D_1,\mathcal D_2\subset X$ and a sequence of maps $S_{n_k}:\mathcal D_2\to X$ such that:
\begin{enumerate}
\item[(1)]  $T^{n_k}(\textbf{x})\to\textbf{0}$ for any
$\textbf{x}\in\mathcal D_1$;

\item[(2)] $S_{n_k}(\textbf{y})\to\textbf{0}$ for any
$\textbf{y}\in\mathcal D_2$;

\item[(3)] $T^{n_k}S_{n_k}(\textbf{y})\to\textbf{y}$ for any
$\textbf{y}\in\mathcal D_2$.
\end{enumerate}
\end{defin}

Note that in the above definition the maps $S_{n_k}$ are not
assumed to be continuous or linear. We will sometimes say that $T$
satisfies the Hypercyclic Criterion with respect to the sequence
$(n_k)$. When it is possible to take $n_k=k$ and $\mathcal
D_1=\mathcal D_2$, it is usually said that $T$ satisfies {\it
Kitai's Criterion} \cite{K}.

\begin{thm}\label{hcthm}
Let $T\in L(X)$, where $X$ is a separable $F$-space. Assume that $T$ satisfies the Hypercyclic
Criterion. Then $T$ is hypercyclic.
\end{thm}

\begin{proof}
According to the Transitivity Theorem it is enough to show that $T$
is topologically transitive. Let $U,V$ be two non-empty open subsets
of $X$. Take $\textbf{x}\in\mathcal D_1\cap U,\ y\in\mathcal D_2\cap
V$. Then $\textbf{x}+S_{n_k}(\textbf{y})\to\textbf{x}\in U$ as
$k\to\infty$. Due to the linearity and the continuity of $T^{n_k}$
we obtain
$T^{n_k}(\textbf{x}+S_{n_k}(\textbf{y}))=T^{n_k}(\textbf{x})+T^{n_k}S_{n_k}(\textbf{y})\to\textbf{y}\in
V$. Hence, for sufficiently large $k$ one gets $T^{n_k}(U)\cap
V\neq\varnothing$. The proof is complete.
\end{proof}

\begin{defin}
Let $T_0:X_0\to X_0$ and $T:X\to X$ be two continuous maps acting on
topological spaces $X_0$ and $X$. The map is said to be a {\bf
quasi-factor} of $T_0$ if there exists a continuous map with dense
range $J:X_0\to X$ such that $TJ=JT_0$. When this can be achieved
with a homeomorphism $J:X_0\to X$, we say that $T_0$ and $T$ are
{\bf topological conjugate}. Finally, when $T_0\in L(X_0)$ and $T\in
L(X)$ and the factoring map (resp. the homeomorphism) $J$ can be
taken as linear, we say that $T$ is a {\bf linear quasi-factor} of
$T_0$ (resp. that $T_0$ and $T$ are {\bf linearly conjugate}).
\end{defin}

The usefulness and importance of these definitions it can be seen in
the following

\begin{lem}\label{hcp} Let $T_0\in L(X_0)$ and $T\in L(X)$. Assume that there exists
a continuous map with dense range  $J:X_0\to X$ such that $TJ=JT_0$. Then the following statements
are satisfied:
\begin{enumerate}
\item[(1)] Hypercyclicity of $T_0$ implies hypercyclicity of $T$;

\item[(2)] Let $J$ be a homeomorphism. If $T_0$ satisfies
Hypercyclic criterion then $T$ satisfies Hypercyclic criterion;

\item[(3)] Let $J$ is linear homeomorphism. Then $T$ is
hypercyclic iff $T_0$ is hypercyclic.
\end{enumerate}
\end{lem}
\begin{proof}
(1) Since $TJ=JT_0$ it is readily to see that $O(J(\textbf{x}_0),T)=J(O(\textbf{x}_0,T_0))$ for any $\textbf{x}_0\in X_0$.
From this and since density of $\mbox{Ran}(J)$ one gets $J(\textbf{x})\in HC(T)$ if $\textbf{x}\in HC(T_0)$.

(2) Now we assume that $T_0$ satisfies Hypercyclic Criterion. Then
$J(\mathcal D_1)$ and $J(\mathcal D_2)$ both are dense sets in $X$
since $J$ has a dense range. For all $\textbf{x}=J(\textbf{x}_0)\in
J(\mathcal D_1)$ we have
$$
T^{n_k}(\textbf{x})=T^{n_k}J(\textbf{x}_0)=JT_0^{n_k}(\textbf{x}_0).
$$
The continuity of $J$ implies that
$T^{n_k}(\textbf{x})\to\textbf{0}$. Denoting by
$\tilde{S}_{n_k}:=JS_{n_k}J^{-1}$,  for every $\textbf{y}\in
J(\mathcal D_2)$ one finds
$$
T^{n_k}\tilde{S}_{n_k}(\textbf{y})=JT_0^{n_k}S_{n_k}J^{-1}(\textbf{y})=JJ^{-1}(\textbf{y})=\textbf{y}
$$
and
$$
\tilde{S}_{n_k}(\textbf{y})=JS_{n_k}J^{-1}(\textbf{y})\to\textbf{0}.
$$
Thus, we have shown that $T$ satisfies Hypercyclic criterion.

Proof of (3) is obvious.
\end{proof}

\begin{rem} Note that if $T\in L(X)$ is hypercyclic and if $J\in L(X)$ has a dense range
and $JT=TJ$ then $HC(T)$ is invariant under $J$.
\end{rem}

We have already observed that if $T$ is a hypercyclic operator on
some $F$-space $X$ then $HC(T)$ is dense $G_\delta$ set of $X$. It
shows that the set $HC(T)$ is large in a topological sense. This
implies largeness in an algebraic sense.

\begin{prop}
Let $T\in L(X)$ be hypercyclic on the separable $F$-space $X$. Then for every $\textbf{x}\in X$
there exist $\textbf{y},\textbf{z}\in HC(T)$ such that $\textbf{x}=\textbf{y}+\textbf{z}$.
\end{prop}

\begin{proof} According to Corollary \ref{corBt} $HC(T)$ is $G_\delta$ set. It follows
that $X\setminus{HC(T)}$ and $X\setminus(\textbf{x}-HC(T))$ are the
first category sets. Then by Baire category
theorem they have non-empty intersection.
\end{proof}

We say that a linear subspace $E\subset X$ is a {\it hypercyclic
manifold} for $T$ if $E\setminus\{\textbf{0}\}$ consists of
entirely of hypercyclic vectors.

\begin{lem}\label{Poly}
Let $T\in L(X)$ and $E\subset X$ be a closed $T$-invariant
subspace. Then either $E=X$ or $E$ has infinite codimension in
$X$.
\end{lem}

\begin{proof}
Assume that $\mbox{dim}(X/E)<\infty$. Let $q:X\to X/E$ be the canonical quotient map.
Since $T$-invariance of $E$ we get $\mbox{Ker}(q)\subset\mbox{Ker}(qT)$. Therefore, one
can find an operator $A\in L(X/E)$ such that $Aq=qT$. Since $q$ is continuous onto, the operator
$A$ is a quasi-factor of $T$. According to the Lemma \ref{hcp} $A$ is hypercyclic on $X/E$.
Since $\mbox{dim}(X/E)<\infty$ by Proposition \ref{findim}, it follows that $X/E=\{\textbf{0}\}$, i.e. $E=X$.
\end{proof}

\begin{lem}\label{Poly1}
Let $T\in L(X)$ be hypercyclic. For any non-zero polynomial $P$,
the operator $P(T)$ has a dense range.
\end{lem}
\begin{proof}
Let $P$ be a non-zero polynomial and
$E:=\overline{\mbox{Ran}(P(T))}$. For any $\textbf{x}\in E$ there
exists a sequence $(\textbf{x}_n)\subset X$ such that
$P(T)\textbf{x}_n\to\textbf{x}$. Then from
$P(T)T(\textbf{x}_n)=TP(T)\textbf{x}_n\to T(\textbf{x})\in E$ we
conclude that $E$ is $T$-invariant. Hence, by Lemma \ref{Poly} it
is enough to show that $\mbox{dim}(X/E)<\infty$.

Let $\textbf{x}\in HC(T)$ and $q:X\to X/E$ be the canonical quotient map.
By the division algorithm and the commutativity of the algebra $\mathbb K[T]$, one can easily see that
$$
\mathbb K[T]\textbf{x}\subset\mbox{Ran}(P(T))+\mbox{span}\{T^i(\textbf{x}): i<\deg(P)\}.
$$
From this we have $q(\mathbb K[T]\textbf{x})$ is
finite-dimensional. Since the cyclicity of $\textbf{x}$ one has
$X/E=q(X)$ is finite-dimensional.
\end{proof}

\begin{thm}
Let $X$ be a topological vector space, and $T\in L(X)$ be hypercyclic.
If $\textbf{x}\in HC(T)$ then $\mathbb K[T]\textbf{x}$ is a hypercyclic manifold for $T$.
In particular, $T$ admits a dense hypercyclic manifold.
\end{thm}

\begin{proof}
Let $\textbf{x}\in HC(T)$ and $P$ be non-zero polynomial. According to the Lemma \ref{Poly1} operator
$P(T)$ has dense range and it commutes with $T$. By the Lemma \ref{hcp} one can gets $P(T)\textbf{x}\in HC(T)$.
This means that $\mathbb K[T]$ is a hypercyclic manifold for $T$. Density of
$\mathbb K[T]$ follows from $O(\textbf{x},T)\subset\mathbb K[T]$.
\end{proof}
We now turn to the supercyclic analogues of Theorems \ref{Btthm} and \ref{hcthm}.

\begin{thm}\label{Btthm2}
Let $X$ be a separable $F$-space, and $T\in L(X)$. The following
statements are equivalent:
\begin{enumerate}
\item[(i)] $T$ is supercyclic;

\item[(ii)] For each pair of non-empty open sets $(U,V)\subset X$
there exist $n\in\mathbb N$ and $\lambda\in\mathbb K$ such that
$\lambda T^n(U)\cap V\neq\varnothing$.
\end{enumerate}
\end{thm}

The proof is similar to the proof of Theorem \ref{Btthm}.

\begin{defin}\label{scdef} Let $X$ be a topological vector space, and let $T\in L(X)$. We say that $T$ satisfies the
{\bf Supercyclic Criterion} if there exist an increasing sequence of integers $(n_k)$, two dense sets
$\mathcal D_1,\mathcal D_2\subset X$ and a sequence of maps $S_{n_k}:\mathcal D_2\to X$ such that:
\begin{enumerate}
\item[(1)]  $\parallel T^{n_k}(\textbf{x})\parallel\parallel
S_{n_k}(\textbf{y})\parallel\to0$ for any $\textbf{x}\in\mathcal
D_1$ and any $\textbf{y}\in\mathcal D_2$;

\item[(2)] $T^{n_k}S_{n_k}(\textbf{y})\to\textbf{y}$ for any
$\textbf{y}\in\mathcal D_2$.
\end{enumerate}
\end{defin}
\begin{thm}\label{scthm} Let $T\in L(X)$, where $X$ is a separable Banach space. Assume that $T$ satisfies the Supercyclic
Criterion. Then $T$ is supercyclic.
\end{thm}
\begin{proof}
Let $U$ and $V$ be two non-empty open subsets of $X$. Take
$\textbf{x}\in\mathcal D_1\cap U$ and $\textbf{y}\in\mathcal
D_2\cap V$. It follows from part (1) of Definition \ref{scdef} and
according to Lemma \ref{x_ny_nlambda_n} that we can find a
sequence of non-zero scalars $(\lambda_k)$ such that
$\lambda_kT^{n_k}(\textbf{x})\to\textbf{0}$ and
$\lambda_k^{-1}S_{n_k}(\textbf{y})\to\textbf{0}$. Then, for large
enough $k$, the vector
$\textbf{z}=\textbf{x}+\lambda_k^{-1}S_{n_k}(\textbf{y})$ belongs
to $U$ and $\lambda_k T^{n_k}(\textbf{z})$ belongs to $V$. By
Theorem \ref{Btthm2} this shows that $T$ is supercyclic.
\end{proof}
\begin{lem}\label{scp}
Let $X_0$ and $X$ be Banach spaces over the field $\mathbb K$ and $T_0\in L(X_0),\ T\in L(X)$ be
such that there exists a $J\in L(X_0,X)$ which has dense range and satisfying
$TJ=JT_0$. Then supercyclicity (cyclicity) of $T_0$ implies supercyclicity (cyclicity)
of $T$
\end{lem}
\begin{proof}
Observe that
$$
\begin{array}{ll}
\{\lambda (T^nJ)(\textbf{x}_0): n\in\mathbb Z_+,\ \lambda\in\mathbb K\}=J(\{\lambda T_0^n(\textbf{x}_0): n\in\mathbb Z_+,\ \lambda\in\mathbb K\}),\\[3mm]
\mbox{span}\{(T^nJ)(\textbf{x}_0): n\in\mathbb Z_+\}=J(\mbox{span}\{T_0^n(\textbf{x}_0): n\in\mathbb Z_+\})
\end{array}
$$
for any $\textbf{x}_0\in X_0$.
Hence, $J(\textbf{x}_0)$ is a supercyclic (cyclic) vector for $T$ for each $\textbf{x}_0\in SC(T_0)$ (resp. $\textbf{x}_0\in C(T)$).
\end{proof}

\section{backward shifts on $c_0$}

In the present section, we are going to study the backward shifts
on $c_0$. Here $c_0$ stands for the set of all sequences which
tend to zero equipped with a norm
$$
\parallel\textbf{x}\parallel:=\sup_{n}\{|x_n|\},\ \ \ \ \textbf{x}\in c_0.
$$
It is clear that $c_0$ is a Banach space.
For convenience, we denote
$$
c_0(\mathbb Z):=\{(x_n)_{n\in\mathbb Z}: x_n\in\mathbb K, |x_{\pm n}|\to0\ \mbox{as }n\to+\infty\}
$$
and
$$
c_0(\mathbb N):=\{(x_n)_{n\in\mathbb N}: x_n\in\mathbb K, |x_{n}|\to0\ \mbox{as }n\to+\infty\}
$$
In what follows, we always assume that $c_0$ is a separable space.
Note that the separability of $c_0$ is equivalent to the
separability of $\mathbb K$. Let $K$ be a countable dense subset
of $\mathbb K$. Then the countable set
$$
c_{00}(\mathbb Z):=\{\lambda_{-n}\textbf{e}_{-n}+\lambda_{-n+1}\textbf{e}_{-n+1}+\cdots+\lambda_{n}\textbf{e}_n:\lambda_{\pm n}\in K, n\in\mathbb N\}
$$
is dense in $c_0(\mathbb Z)$, where $\textbf{e}_n$ is an unit vector such that only $n$-th coordinate equals to 1 and others are zero.

Let $\textbf{a}=(a_n)_{n\in\mathbb Z}$ be a bounded sequence of
non-zero numbers of $\mathbb K$. An operator $B_\textbf{a}$ on
$c_0(\mathbb Z)$ defined by
$B_\textbf{a}(\textbf{e}_n)=a_{n}\textbf{e}_{n-1}$ is called {\it
bilateral weighted backward shift} if $a_i\neq1$ for some
$i\in\mathbb Z$, otherwise it is called {\it bilateral unweighted
backward shift} and we denote it by $B$.

\begin{thm}
Let $B_\textbf{a}$ be a bilateral weighted backward shift operator
on $c_0(\mathbb Z)$. Then the following statements hold:
\begin{enumerate}
\item[(i)] $B_\textbf{a}$ is hypercyclic if and only if, for any
$q\in\mathbb N$,
\begin{equation}\label{hhh}
\liminf\limits_{n\to+\infty}\max\left\{\prod\limits_{i=1}^{n+q}|a_i^{-1}|,\prod_{j=1}^{n-q}|a_{-j+1}|\right\}=0.
\end{equation}
\item[(ii)] $B_\textbf{a}$ is supercyclic if and only if, for any
$q\in\mathbb N$,
\begin{equation}\label{ccc}
\liminf\limits_{n\to+\infty}\prod\limits_{i=1}^{n+q}|a_i^{-1}|\times\prod_{j=1}^{n-q}|a_{-j+1}|=0.
\end{equation}
\end{enumerate}
\end{thm}

\begin{proof} For any weight $\textbf{b}\in l_\infty(\mathbb Z)$ with $b_n\neq0,\ n=0,\pm1,\pm2,\dots$
we introduce the weighted space
$$
c_0(\mathbb Z,\textbf{b}):=\left\{\textbf{x}\in c_0(\mathbb Z):\ \parallel\textbf{x}\parallel_\textbf{b}=\sup_{n}|b_nx_n|\right\}.
$$
Take a weight sequence $\textbf{b}=(b_n)_{n\in\mathbb Z}$ as
follows $b_0=1$ and $b_nb^{-1}_{n+1}=a_{n+1}$. Let $B$ be the
bilateral backward shift on $c_0(\mathbb Z,\textbf{b})$. Then
$B_\textbf{a}$ is linearly conjugate to $B$. Indeed, the operator
$J:c_0(\mathbb Z)\to c_0(\mathbb Z,\textbf{b})$ defined by
$(J\textbf{x})_n=b_n^{-1}x_n$ is a linear homeomorphism and
$J(c_0(\mathbb Z))=c_0(\mathbb Z,\textbf{b}),\ JB_\textbf{a}=BJ$.
According to Lemma \ref{hcp} (resp. Lemma \ref{scp})
hypercyclicity (supercyclicity) of $B_\textbf{a}$ is equivalent to
the hypercycility (resp. supercyclicity) of $B$.

Assume that $B$ is hypercyclic and fix $q\in\mathbb N$. Due to the
density of $O(\textbf{x},B)$ (for all $\textbf{x}\in HC(B)$), for
an arbitrary $\varepsilon>0$ one can find $\textbf{x}\in HC(B)$
and an integer $n>2q$ such that
$$
\parallel \textbf{x}-\textbf{e}_q\parallel_\textbf{b}<\varepsilon\ \ \ \ \mbox{and}\ \ \ \ \parallel B^n(\textbf{x})-\textbf{e}_q\parallel_\textbf{b}<\varepsilon.
$$
These inequalities imply that
\begin{equation}\label{qn+q}
|b_q(x_q-1)|<\varepsilon,\ \ \ \ \ |b_{n+q}x_{n+q}|<\varepsilon,
\end{equation}
\begin{equation}\label{q-n+q}
|b_q(x_{n+q}-1)|<\varepsilon,\ \ \ \ \
|b_{-n+q}x_{q}|<\varepsilon.
\end{equation}
We assume that $\varepsilon<|b_q|$. Then from the first
inequalities of \eqref{qn+q} and \eqref{q-n+q} we obtain
$|x_q-1|<1$ and $|x_{n+q}-1|<1$. Hence, by the non-Archimedean
norm's property, one gets $|x_q|=|x_{n+q}|=1$. Putting it into the
second inequalities of \eqref{qn+q} and \eqref{q-n+q} one finds
$|b_{\pm n+q}|<\varepsilon$, which is equivalent to
\begin{equation}\label{hhhb}
\liminf_{n\to+\infty}|b_{\pm n+q}|=0.
\end{equation}
Now let us assume that \eqref{hhhb} holds for any $q\in\mathbb N$. We will show that
$B$ satisfies the Hypercyclic Criterion. Take a some positive number $M$ such that
$$
M>\max\left\{1,\sup_n\frac{|b_n|}{|b_{n+1|}}\right\}.
$$
By \eqref{hhhb}, one can find an increasing sequence of positive integers $\{n_k\}$ such that
$$
|b_{\pm n_{k}+k}|\leq M^{-3k}\ \ \ \ \mbox{for all}\ \ k\in\mathbb N.
$$
Assume that $i$ be a fixed integer and $k>|i|$. Then $\left|b_{\pm n_k+i}\right|<M^{i+k}\left|b_{\pm n_k+k}\right|\leq M^{-2k+i}<M^{-k}$. It follows
that $b_{n_k+i}\to0$ as $k\to\infty$ for any $i\in\mathbb Z$. Now, let $\mathcal
D_1=\mathcal D_2:=c_{00}(\mathbb Z)$ and let $S$ be the forward
shift, defined on $\mathcal D_2$ by
$S(\textbf{e}_i)=\textbf{e}_{i+1}$. Due to the linearity of $B$
and $S$, it is enough to show that $B^{n_k}(\textbf{e}_i)\to0$ and
$S^{n_k}(\textbf{e}_i)\to0$ for any $i\in\mathbb Z$. But this is
clear since
$$
\parallel B^{n_k}(\textbf{e}_i)\parallel_\textbf{b}=|b_{-n_k+i}|\ \ \ \ \mbox{and}\ \ \ \ \parallel S^{n_k}(\textbf{e}_i)\parallel_\textbf{b}=|b_{n_k+i}|.$$
Thus, we have shown that $B_\textbf{a}$ is hypercyclic if and only
if for any $q\in\mathbb N$ holds \eqref{hhhb}. According to
$$
b_n=\prod_{i=1}^na_i^{-1}\ \ \ \mbox{and}\ \ \ b_{-n}=\prod_{j=1}^na_{-j+1}\ \ \ \mbox{for all}\ n\in\mathbb N
$$
it is easy to see that \eqref{hhhb} and \eqref{hhh} are equivalent.

Now we turn to the supercyclic case. Suppose that $B$ is
supercyclic and $q\in\mathbb N$. Let $\varepsilon<0$  be an
arbitrary number. Then the density of supercyclic vectors implies
the existence of $\textbf{x}\in c_0(\mathbb Z,\textbf{b}),\
\lambda\in\mathbb K^\times$ and $n>2q$ such that
$$
\parallel\textbf{x}-\textbf{e}_q\parallel_\textbf{b}<\varepsilon\ \ \ \ \mbox{and}\ \ \ \ \parallel\lambda B^n(\textbf{x})-\textbf{e}_q\parallel_\textbf{b}<\varepsilon.
$$
As above, we obtain
$$
\begin{array}{ll}
|b_q(x_q-1)|<\varepsilon,&\ \ \ \ \ |b_{n+q}x_{n+q}|<\varepsilon,\\[2mm]
|b_q(\lambda x_{n+q}-1)|<\varepsilon,&\ \ \ \ \ |\lambda b_{-n+q}x_{q}|<\varepsilon.
\end{array}
$$
Assuming $\varepsilon<|b_q|$ and using the non-Archimedean norm's
property one finds
$$
|b_{-n+q}|<\frac{\varepsilon}{|\lambda|}\ \ \ \ \mbox{and}\ \ \ \ \ |b_{n+q}|<\varepsilon|\lambda|.
$$
Hence, $|b_{-n+q}b_{n+q}|<\varepsilon^2$ which yields
\begin{equation}\label{cccb}
\liminf\limits_{n\to+\infty}|b_{n+q}b_{-n+q}|=0.
\end{equation}
Note that \eqref{cccb} and \eqref{ccc} are equivalent.

If the condition \eqref{cccb} holds then we can find as above an
increasing sequence $(n_k)$ such that, for any $i,j\in\mathbb Z$,
$$
b_{n_k+i}b_{-n_k+j}\to0,\ \ \ \mbox{as}\ k\to+\infty.
$$
This  exactly means that the Supercyclic Criterion is satisfied
for $\mathcal D_1=\mathcal D_2:=c_{00}(\mathbb Z)$ and the forward
shift $S$, since
$$
\parallel B^{n_k}(\textbf{e}_j)\parallel_\textbf{b}\cdot\parallel S_{n_k}(\textbf{e}_i)\parallel_\textbf{b}=|b_{n_k+i}b_{-n_k+j}|.
$$
This completes the proof.
\end{proof}

From this theorem we immediately find the following results.

\begin{cor}
Let $B_\textbf{a}$ be a bilateral weighted backward shift on $c_0(\mathbb Z)$. If $B_\textbf{a}$ is supercyclic
then $\lambda B_\textbf{a}$ is supercyclic for any $\lambda\in\mathbb K^\times$.
\end{cor}
\begin{cor}
Let $B_\textbf{a}$ be a bilateral weighted backward shift on $c_0(\mathbb Z)$. If the weight sequence $\textbf{a}=(a_n)_{n\in\mathbb Z}$
is symmetrical to the norm, i.e. $|a_n|=|a_{-n}|,\ n=1,2,\dots$ then $B_\textbf{a}$ is not supercyclic.
\end{cor}
\begin{cor}
Let $B$ be a bilateral unweighted backward shift on $c_0(\mathbb Z)$. Then $B$ is not supercyclic. Moreover,
$\lambda B$ is not supercyclic for any $\lambda\in\mathbb K$.
\end{cor}
\begin{cor}\label{corol1}
Let $\textbf{a}$ and $\textbf{b}$ be weighted sequences such that
$|a_n|>|b_n|$ for any $n\in\mathbb Z$. Then
$B_{\textbf{a}+\textbf{b}}$ is hypercyclic (resp. supercyclic) if
and only if $B_\textbf{a}$ is hypercyclic (resp. supercyclic).
\end{cor}
\begin{proof} By non-Archimedean norm's property
we have $|a_n+b_n|=|a_n|$ for any $n\in\mathbb Z$. Using it to
\eqref{hhh} (resp. \eqref{ccc}) we can conclude that hypercyclicity (supercyclicity)
of $B_{\textbf{a}}$ and $B_{\textbf{a}+\textbf{b}}$ are equivalent.
\end{proof}

\begin{rem} Note that in the real case Corollary \ref{corol1} is not
true.
\end{rem}

Now let us consider a unilateral weighted backward shifts on
$c_0(\mathbb N)$. Recall that operator defined as
$B_\textbf{a}(\textbf{e}_1)=0$ and
$B_\textbf{a}(\textbf{e}_n)=a_{n-1}\textbf{e}_{n-1}$ if $n\geq2$,
is called {\it unilateral} weighted backward shift. Here
$\textbf{a}=(a_n)_{n\in\mathbb N}$ be a bounded sequence of
non-zero numbers of $\mathbb K$. The operator $B_\textbf{a}$ is
called unilateral unweighted backward shift if $a_n=1$ for all
$n\geq1$. We denote by $B$ a unilateral unweighted backward shift
operator.

\begin{thm} Any unilateral weighted backward shift $B_\textbf{a}$ on $c_0(\mathbb N)$ is supercyclic. $B_\textbf{a}$
is hypercyclic iff
\begin{equation}\label{uwhc}
\limsup_{n\to\infty}\prod_{i=1}^n|a_i|=\infty.
\end{equation}
\end{thm}
\begin{proof}
Let $B_\textbf{a}$ be a unilateral weighted backward shift. Let
$\mathcal D_1=\mathcal D_2:= c_{00}(\mathbb N)$ be the set of all
finitely supported sequences. Let $S_\textbf{a}$ be the linear map
defined on $\mathcal D_2$ by
$S_\textbf{a}(\textbf{e}_n)=a_{n}^{-1}\textbf{e}_{n+1}$ and, for
each $k\in\mathbb N$, set $S_k:=S_\textbf{a}^k$. Then, the
Supercyclicity Criterion is satisfied with respect to $k$ because
$\parallel B_\textbf{a}^k(\textbf{x})\parallel=0$ for large enough
$k$ and $B_\textbf{a}^kS_k=I$ on $\mathcal D_2$. According to
Theorem \ref{scthm} operator $B_\textbf{a}$ is supercyclic.

Now we are going to establish that hypercyclicity of
$B_\textbf{a}$ equivalent to \eqref{uwhc}. Suppose first that
\eqref{uwhc} holds, and let us show that $B_\textbf{a}$ satisfies
the Hypercyclicity Criterion. It is enough to show that
$S_k(\textbf{x})\to\textbf{0}$ as $k\to\infty$ for all
$\textbf{x}\in c_{00}(\mathbb N)$. Let $\textbf{x}\in
c_{00}(\mathbb N)\setminus\{\textbf{0}\}$. Then there exists
positive integer $q$ such that $x_q\neq0$ and $x_m=0$ for all
$m>q$. Denote $x_j^{(k)}:=(S_k\textbf{x})_j,\ j=1,2,3,\dots$. We
have $x_j^{(k)}=0$ if $j=\overline{1,k}$ or $j>q+k$, and
$$
x_{j+k}^{k}=\frac{x_j}{\prod\limits_{i=1}^{k}|a_{j+i-1}|},\ \ \ j=\overline{1,q}.
$$
From \eqref{uwhc} we obtain $x_{j+k}^{k}\to0$ as $k\to\infty$.

Let us assume that $B_\textbf{a}$ is hypercyclic, and take an
arbitrary number $\varepsilon>0$. Then the density of hypercyclic
vectors implies the existence of $\textbf{x}\in c_0(\mathbb N)$
and an integer $k>2$ such that
$$
\parallel\textbf{x}-\textbf{e}_1\parallel<\varepsilon\ \ \ \mbox{and}\ \ \ \parallel B_\textbf{a}^k(\textbf{x})-\textbf{e}_1\parallel<\varepsilon.
$$
From these relations, we obtain $|x_{k+1}|<\varepsilon$ and
$\left|\prod_{i=1}^ka_ix_{k+1}-1\right|<\varepsilon$. Again using
the non-Archimedean norm's property from the last inequalities one
finds
$$
\prod_{i=1}^k\left|a_i\right|=\frac{1}{|x_{k+1}|}>\frac{1}{\varepsilon}.
$$
The arbitrariness of $\varepsilon$ yields \eqref{uwhc}. The proof
is complete.
\end{proof}

\begin{cor}
Let $B$ be an unilateral unweighted backward shift on $c_0(\mathbb
N)$. Then the following assertions hold:
\begin{enumerate}
\item[(i)]  An operator $\lambda B$ is supercyclic for any
$\lambda\in\mathbb K^\times$;

\item[(ii)] $\lambda B$ is hypercyclic iff
$$
\limsup_{n\to\infty}\prod_{i=1}^n|\lambda a_i|=\infty.
$$
\end{enumerate}
\end{cor}

\section{$\lambda\textbf{I}+\mu\textbf{B}$ operators on $c_0$}

In this section we are going to consider the following operator
$$
T_{\lambda,\mu}=\lambda I+\mu B,
$$
where $I$ is a identity and $B$ is unweighted backward shift. We
will show that there does not exist pair of $(\lambda,\mu)$ such
that $T_{\lambda,\mu}$ can be supercyclic on $c_0(\mathbb Z)$.
But, for any pair of $(\lambda,\mu)$ with $|\lambda|<|\mu|$ an
operator $T_{\lambda,\mu}$ is supercyclic on $c_0(\mathbb N)$.
Moreover, we will prove that the condition $|\lambda|<|\mu|$ is
necessary for the supercyclicity of $T_{\lambda,\mu}$ on
$c_0(\mathbb N)$.

\begin{thm} The operator $T_{\lambda,\mu}$ on $c_0(\mathbb Z)$ is not supercyclic for all $\lambda,\mu\in\mathbb K$.
\end{thm}

\begin{proof} First, we consider the case $|\lambda|\geq|\mu|$. Take $\textbf{x}\in c_0(\mathbb Z)\setminus\{\textbf{0}\}$. Then there exists
a number $k\in\mathbb N$ such that
$|x_k|>|x_m|$ for all $m>k$.

Denote
$$
x_i^{(n)}:=\left(T_{\lambda,\mu}^n\textbf{x}\right)_i,\ \ \ \ i=0,\pm1,\pm2,\dots
$$
It is easy to get the following recurrence formula
$$
x_i^{(n)}=\lambda^n\sum_{j=0}^n \binom{n}{j}
\mu^j\lambda^{-j}x_{i+j},\ \ \ i=0,\pm1,\pm2,\dots
$$
Due to $| \binom{n}{j}|\leq1,\ j=\overline{1,n}$ and
$|\mu|\leq|\lambda|$, and using the non-Archimedean norm's
property one gets
$$
\left|x_k^{(n)}\right|=\left|\lambda^nx_k\right|
$$
and
$$
\left|x_{k+1}^{(n)}\right|<\left|\lambda^nx_k\right|.
$$
Then for any $\alpha\in\mathbb K^\times$ we get
$$
\left|\alpha x_k^{(n)}\right|>\left|\alpha x_{k+1}^{(n)}\right|.
$$
From this using non-Archimedean norm's property we immediately obtain
$$
\parallel\alpha T^n_{\lambda,\mu}(\textbf{x})-\textbf{e}_{k+1}\parallel\geq1.
$$
It yields that $O(\textbf{x},\alpha T_{\lambda,\mu})\cap
B(\textbf{e}_{k+1},1)=\varnothing$. The arbitrariness of $\alpha$
implies that $c_0(\mathbb Z)\setminus\overline{\mathbb K\cdot
O(\textbf{x},T_{\lambda,\mu})}\neq\varnothing$. Since $\textbf{x}$
is an arbitrary vector we conclude that $T_{\lambda,\mu}$ can not
be supercyclic if $|\lambda|\geq|\mu|$.

Now we assume that $|\mu|>|\lambda|$. Pick a non-zero vector
$\textbf{y}$. We can take an integer number $\ell$ such that
$|y_{\ell}|\geq|y_{i}|$ for all $i>\ell$ and $|y_\ell|>|y_j|$ for
all $j<\ell$. Then for any $k\in\mathbb Z$ we have
\begin{equation}\label{y_k-n}
y_{k-n}^{(n)}=\mu^n\sum_{j=0}^n
\binom{n}{j}\lambda^{n-j}\mu^{j-n}y_{k-n+j}.
\end{equation}
Using the strong triangle inequality one gets
$$
\left|y_{\ell-n}^{(n)}\right|=\left|\mu^ny_{\ell}\right|
$$
Pick an integer number $m$ such that $|y_i|<|y_\ell|$ for all
$i>m$. Then from \eqref{y_k-n} for any $i>m$ we obtain
$$
\left|y_{i}^{(n)}\right|<\left|\mu^ny_{\ell}\right|.
$$
Since
$\left|y_{\ell-n}^{(n)}\right|>\left|y_{m+\ell}^{(n)}\right|$ for
any $\beta\in\mathbb K$ inequality $\left|\beta
y_{\ell-n}^{(n)}\right|<1$ implies $\left|\beta
y_{m+1}^{(n)}-1\right|=1$. It yields that $O(\textbf{y},\beta
T_{\lambda,\mu})\cap B(\textbf{e}_{m+1},1)=\varnothing$ for any
$\beta\in\mathbb K$. Since $\textbf{y}$ is an arbitrary vector we
conclude that $T_{\lambda,\mu}$ can not be supercyclic on
$c_0(\mathbb Z)$ if $|\lambda|<|\mu|$. This completes the proof.
\end{proof}

From Remark \ref{remsubset} we obtain the following

\begin{cor}
The operator $T_{\lambda,\mu}$ on $c_0(\mathbb Z)$ is not
hypercyclic for all $\lambda,\mu\in\mathbb K$.
\end{cor}

Now we consider the operator
$T_{\lambda,\mu}$ on $c_0(\mathbb N)$. We will show that
hypercyclicity of $T_{\lambda,\mu}$ is equivalent to the Hypercyclic Criterion.

\begin{thm}\label{theqhc} For the the operator $T_{\lambda,\mu}$ acting on $c_0(\mathbb N)$ the following statements are
equivalent:
\begin{enumerate}
\item[(i)] $T_{\lambda,\mu}$ satisfies Hypercyclic Criterion;

\item[(ii)] $T_{\lambda,\mu}$ is hypercyclic;

\item[(iii)] $|\lambda|<1<|\mu|$.
\end{enumerate}
\end{thm}
To prove the theorem we first prove three auxiliary lemmas.

\begin{lem}\label{lem1}
If the operator $T_{\lambda,\mu}$ acting on $c_0(\mathbb N)$ is
hypercyclic then $|\mu|>|\lambda|$ and $|\mu|>1$.
\end{lem}

\begin{proof}
Assume that $T_{\lambda,\mu}$ be hypercyclic. We immediately get
that $\parallel T_{\lambda,\mu}\parallel>1$. Using the
non-archimedean norm's property one finds
$\max\{|\lambda|,|\mu|\}>1$. Let us suppose that
$|\mu|\leq|\lambda|$. Take $\textbf{x}\in HC(T_{\lambda,\mu})$.
Since the vector $\textbf{x}$ is not zero, then there exists a
number $k\in\mathbb N$ such that $|x_k|>|x_m|$ for all $m>k$.

It is easy to get the following recurrence formula
$$
x_i^{(n)}=\lambda^n\sum_{j=0}^n\binom{n}{j}\left(\frac{\mu}{\lambda}\right)^jx_{i+j},\
\ \ i=1,2,3,\dots
$$
From $|\binom{n}{j}|\leq1,\ j=\overline{1,n}$ and
$|\mu|\leq|\lambda|$, by means of the non-Archimedean norm's
property one gets
$$
\left|x_k^{(n)}\right|=\left|\lambda^nx_k\right|.
$$
From $|\lambda|>1$ we get $\left|x_k^{(n)}\right|>|x_k|$.
Hence,
$$
\parallel T_{\lambda,\mu}^n(\textbf{x})\parallel>|x_k|>0.
$$
Then $O(\textbf{x},T_{\lambda,\mu})\cap
B(0,\varepsilon)=\varnothing$ for any positive
$\varepsilon<|x_k|$. This means that $\textbf{x}\not\in
HC(T_{\lambda,\mu})$. Thus, we have shown that $T_{\lambda,\mu}$
cannot be hypercyclic if $|\mu|\leq|\lambda|$. From this fact and
$\max\{|\lambda|,|\mu|\}>1$ we get $|\mu|>1$.
\end{proof}

\begin{lem}\label{lem2} Let $|\mu|>1$. Then $T_{\lambda,\mu}$ acting on $c_0(\mathbb N)$ is hypercyclic if $|\lambda|<1$.
\end{lem}

\begin{proof} Let $|\lambda|<1<|\mu|$. We define the operator $S_{\mu,\lambda}$ as follows
\begin{equation}\label{S_m,l}
\begin{array}{ll}
\left(S_{\mu,\lambda}\textbf{x}\right)_1=0\\[2mm]
\left(S_{\mu,\lambda}\textbf{x}\right)_i=\frac{1}{\mu}\left(\sum\limits_{j=1}^{i-1}\left(\frac{-\lambda}{\mu}\right)^{j-1}x_{i-j}\right), \ \ \ \ i=2,3,4,\dots
\end{array}
\end{equation}
Then one has $T_{\lambda,\mu}S_{\mu,\lambda}=I$. Let
$\textbf{x}\in c_{00}$. It is clear that
$T_{\lambda,\mu}^n(x)\to0$ as $n\to\infty$. It follows from the
strong triangle inequality that
$$\parallel S_{\mu,\lambda}^n(\textbf{x})\parallel\leq\frac{1}{|\mu^n|}\parallel\textbf{x}\parallel.$$
Since $|\mu|>1$ we obtain $S_{\mu,\lambda}^n(\textbf{x})\to0$ as
$n\to\infty$. Hence, the operator $T_{\lambda,\mu}$ satisfies the
Hypercyclic Criterion, therefore, Theorem \ref{hcthm} implies that
$T_{\lambda,\mu}$ is a hypercyclic.
\end{proof}

\begin{lem}\label{lem3} Let $1\leq|\lambda|<|\mu|$. Then $T_{\lambda,\mu}$ acting on $c_0(\mathbb N)$ is not hypercyclic.
\end{lem}

\begin{proof}
Let $1\leq|\lambda|<|\mu|$. Denote $r=\frac{|\lambda|}{|\mu|}$. Take an open subset $U$ of $c_0(\mathbb N)$ as follows:
$$
U:=\left\{\textbf{x}\in c_0(\mathbb N): r^{3^k+2}<|x_k|<r^{3^{k}},\ k\in\mathbb N\right\}.
$$
Let $\textbf{x}\in U$. We need compute the norm of $\lambda x_k+\mu x_{k+1},\ k\geq1$. Since
$$
\begin{array}{ll}|x_k|>r^{3^k+2},\\[3mm]
\left|\mu\lambda^{-1}x_{k+1}\right|=r^{-1}|x_{k+1}|<r^{3^{k+1}-1}
\end{array}
$$
and using the non-Archimedean norm's property we obtain
$$
|\lambda x_k+\mu x_{k+1}|=|\lambda x_k|\ \ \ \mbox{for all }\ k\in\mathbb N.
$$
From this one gets $T_{\lambda,\mu}(U)\subset \lambda U$. Hence,
$T_{\lambda,\mu}^n(U)\subset \lambda^n U$ for any $n\in\mathbb N$.
Pick an open ball $V:=\{\textbf{x}\in c_0(\mathbb N):
|x-e_2|<1\}$. We will show that $T_{\lambda,\mu}^n(U)\cap
V=\varnothing$ for any $n\in\mathbb N$. Assume that there exists a
positive integer $n$ such that $\lambda^n U\cap V\neq\varnothing$.
Let $\textbf{y}\in\lambda^n U\cap V$. Then we have $|y_1|<1$ and
$|y_2-1|<1$. From the last inequality with the non-Archimedean
norm's property one gets $|y_2|=1$. Due to $\textbf{y}\in\lambda^n
U$ we have $|y_1|>|y_2|=1$. It is a contradiction to $|y_1|<1$.
This contradiction shows that $\lambda^n U\cap V=\varnothing$ for
all $n\geq1$. According to $T_{\lambda,\mu}^n(U)\subset\lambda^nU$
and $\lambda^n U\cap V=\varnothing$ one immediately finds that
$T_{\lambda,\mu}^n(U)\cap V=\varnothing$ for any $n\in\mathbb N$.
This means that $T_{\lambda,\mu}$ is not topologically transitive.
According to Theorem \ref{Btthm} the operator $T_{\lambda,\mu}$ is
not hypercyclic.
\end{proof}

\textit{Proof of Theorem \ref{theqhc}} The implication
(i)$\Rightarrow $(ii) follows from Theorem \ref{hcthm}. Lemmas
\ref{lem1} and \ref{lem3} imply that the implication
(ii)$\Rightarrow $(iii). Finally, (iii)$\Rightarrow $(i) follows
from Lemma \ref{lem2}. The completes the proof.

\begin{rem}
According to Theorem \ref{theqhc} an operator $I+\mu B$ on
$c_0(\mathbb N)$ can not be hypercyclic for any $\mu\in\mathbb K$.
But, in real case \cite{Shk1} it is hypercyclic for
$|\mu|_\infty>1$, where $|\cdot|_\infty$ is absolute value.
\end{rem}

Now we will study supercyclicity of $T_{\lambda,\mu}$. Similarly
to the hypercyclic case we have the following
\begin{thm}
For the operator $T_{\lambda,\mu}$ acting on $c_0(\mathbb N)$ the
following statements are equivalent:
\begin{enumerate}
\item[(i)] $T_{\lambda,\mu}$ satisfies Supercyclic Criterion;
\item[(ii)] $T_{\lambda,\mu}$ is supercyclic;

\item[(iii)] $|\lambda|<|\mu|$.
\end{enumerate}
\end{thm}

\begin{proof}
The implication (i)$\Rightarrow$(ii) follows from Theorem
\ref{scthm}. We will establish the implications
(ii)$\Rightarrow$(iii)$\Rightarrow$(i).

(ii)$\Rightarrow$(iii) Let assume that $T_{\lambda,\mu}$ is a
supercyclic. Suppose that $|\mu|\leq|\lambda|$. Let $\textbf{x}\in
SC(T_{\lambda,\mu})$. Due to  $\textbf{x}\neq\textbf{0}$ there
exists a positive integer $k$ such that $|x_k|>|x_m|$ for all
$m>k$. Denote
$x_i^{(n)}:=\left(T^n_{\lambda,\mu}\textbf{x}\right)_i$. Then for
any $n\geq1$ using the non-Archimedean norm's property we get
\begin{equation}\label{x_k}
\left|x_k^{(n)}\right|=\left|\lambda^nx_k\right|
\end{equation}
and
\begin{equation}\label{x_k+1}
\left|x_{k+1}^{(n)}\right|<\left|\lambda^nx_k\right|
\end{equation}
Since $\textbf{x}$ is a supercyclic vector, there exist a number
$n\in\mathbb N$ and $\alpha\in\mathbb K$ such that
$$
\parallel\alpha T_{\lambda,\mu}^n(\textbf{x})-\textbf{e}_{k+1}\parallel<1.
$$
It follows that
\begin{equation}\label{12ineq}
\left|\alpha x_k^{(n)}\right|<1,\ \ \ \left|\alpha x_{k+1}^{(n)}-1\right|<1
\end{equation}
On the other hand, from \eqref{x_k} and \eqref{x_k+1} we obtain
$$
\left|\alpha x_k^{(n)}\right|=\left|\alpha\lambda^nx_k\right|
$$
and
$$
\left|\alpha x_{k+1}^{(n)}\right|<\left|\alpha\lambda^nx_k\right|.
$$
From these and the non-Archimedean norm's property one finds
$$
\left|\alpha x_k^{(n)}\right|<1,\ \ \ \left|\alpha x_{k+1}^{(n)}-1\right|=1.
$$
It is a contradiction to \eqref{12ineq}. This yields that
$T_{\lambda,\mu}$ can not be supercyclic if $|\mu|\leq|\lambda|$.

(iii)$\Rightarrow$(i) Let $|\lambda|<|\mu|$. Take an arbitrary
vector $\textbf{x}\in c_{00}(\mathbb N)$. Then there exists a
number $\ell\in\mathbb N$ such that $x_\ell\neq0$ and $x_j=0$ for
all $j>\ell$. It is clear that $x_j^{(n)}=0,\ j>\ell$ for any
$n\geq1$. For a given $n\geq1$ we have
$$
x_j^{(n)}=\sum_{i=0}^{\ell-j}\binom{n}{i}\lambda^{n-i}\mu^ix_{j+i},\
\ \ \ \ \ j=\overline{1,l}.
$$
From this one gets
$$
\left|x_j^{(n)}\right|\leq\left|\lambda^{n-\ell+j}\mu^{\ell-j}\right|\cdot|x_j|.
$$
Hence,
\begin{equation}\label{T^n}
\parallel T_{\lambda,\mu}^n(\textbf{x})\parallel\leq\left|\lambda^{n-k+j}\mu^{k-j}\right|\cdot\parallel\textbf{x}\parallel
\end{equation}
Let now pick an arbitrary vector $\textbf{y}\in c_{00}(\mathbb N)$
and compute the norm of $S_{\mu,\lambda}^n(\textbf{y})$, where the
operator $S_{\mu,\lambda}$ is defined by \eqref{S_m,l}. From
\eqref{S_m,l} and using the non-Archimedean norm's property one
finds
\begin{equation}\label{S^n}
\parallel S^n_{\mu,\lambda}(\textbf{y})\parallel\leq\left|\mu^{-n}\right|\cdot\parallel\textbf{y}\parallel.
\end{equation}
Multiplying \eqref{T^n} and \eqref{S^n} we obtain
$$
\parallel T^n_{\lambda,\mu}(\textbf{x})\parallel\cdot\parallel S^n_{\mu,\lambda}(\textbf{y})\parallel\leq
\left(\frac{|\lambda|}{|\mu|}\right)^{n-k+j}\parallel\textbf{x}\parallel\cdot\parallel\textbf{y}\parallel.
$$
Since $|\lambda|<|\mu|$ we have $\parallel
T^n_{\lambda,\mu}(\textbf{x})\parallel\cdot\parallel
S^n_{\mu,\lambda}(\textbf{y})\parallel\to0$ as $n\to\infty$.
Hence, $T_{\lambda,\mu}$ satisfies the Supercyclic Criterion if
$|\lambda|<|\mu|$. This completes the proof.
\end{proof}

\begin{rem} We stress that all operators on $c_0$  considered above are hypercyclic
(resp. supercyclic) if they satisfy Hypercyclic (reps.
Supercyclic) Criterion. It is natural to ask: does there exists a
hypercyclic (resp. supercyclic) linear operator on $c_0$ which
does not satisfy HC (SC)? We conjecture that such kind of linear
operators on $c_0$ do not exist.
\end{rem}

\end{document}